\documentclass[a4paper]{article}
\usepackage{hyperref}
\usepackage{cleveref}
\hypersetup{hypertexnames = false, bookmarksdepth = 2, bookmarksopen = true, colorlinks, linkcolor = black, citecolor = black, urlcolor = black, pdfstartview={XYZ null null 1}}

\usepackage{amsfonts}
\usepackage[fleqn, leqno]{amsmath}
\usepackage{amsthm}
\usepackage[maxbibnames=99]{biblatex}
\usepackage{booktabs}
\usepackage{diagbox}
\usepackage{enumitem}
\usepackage{mathtools}
\usepackage{parskip}
\usepackage{thmtools}
\usepackage{tikz}
\usepackage{tikz-cd}
\usepackage[colorinlistoftodos, textsize = footnotesize]{todonotes}
\usepackage{xparse}
\usepackage{xspace}

\usepackage[T1]{fontenc}
\usepackage{libertine}
\usepackage[libertine]{newtxmath}
\usepackage[scaled=0.83]{beramono}
\usepackage{eucal}
\usepackage{microtype}
\newcounter{todocounter}
\DeclareDocumentCommand\addreference{g}{\stepcounter{todocounter}\todo[color = blue!30]{\thetodocounter. Add reference\IfNoValueF{#1}{: #1}}\xspace}
\DeclareDocumentCommand\checkthis{g}{\stepcounter{todocounter}\todo[color = red!50]{\thetodocounter. Check this\IfNoValueF{#1}{: #1}}\xspace}
\DeclareDocumentCommand\fixthis{g}{\stepcounter{todocounter}\todo[color = orange!50]{\thetodocounter. Fix this\IfNoValueF{#1}{: #1}}\xspace}
\DeclareDocumentCommand\expand{g}{\stepcounter{todocounter}\todo[color = green!50]{\thetodocounter. Expand\IfNoValueF{#1}{: #1}}\xspace}

\declaretheoremstyle[
  spaceabove = 3pt,
  spacebelow = 3pt,
]{lecture}
\theoremstyle{lecture}
\newtheorem{theorem}{Theorem}

\newtheorem{corollary}[theorem]{Corollary}
\newtheorem{definition}[theorem]{Definition}

\newtheorem{lemma}[theorem]{Lemma}
\newtheorem{proposition}[theorem]{Proposition}
\newtheorem{remark}[theorem]{Remark}

\crefname{figure}{figure}{figures}

\makeatletter
\def\gitfootnote{\gdef\@thefnmark{}\@footnotetext}
\makeatother

\mathchardef\mhyphen="2D
\newcommand\dash{\nobreakdash-\hspace{0pt}}

\let\oldbigwedge\bigwedge
\renewcommand\bigwedge{\oldbigwedge\nolimits}

\newcommand\bounded{\ensuremath{\mathrm{b}}}

\newcommand\moduli{\ensuremath{\mathrm{M}}}
\newcommand\RR{\ensuremath{\mathrm{R}}}
\newcommand\RRR{\ensuremath{\mathbf{R}}}
\newcommand\tangent{\ensuremath{\mathrm{T}}}
\newcommand\todd{\ensuremath{\mathrm{td}}}

\DeclareMathOperator\adjoint{ad}
\DeclareMathOperator\cc{c}
\DeclareMathOperator\chern{ch}
\DeclareMathOperator\codim{codim}
\DeclareMathOperator\derived{\mathbf{D}}
\DeclareMathOperator\Ext{Ext}
\DeclareMathOperator\HH{H}
\DeclareMathOperator\hh{h}
\DeclareMathOperator\HP{HP}
\DeclareMathOperator\Hom{Hom}
\DeclareMathOperator\identity{id}

\DeclareMathOperator\Pic{Pic}
\DeclareMathOperator\sheafEnd{\mathcal{E}nd}
\DeclareMathOperator\Spec{Spec}
\DeclareMathOperator\Sym{Sym}

\title{Admissible subcategories in derived categories of moduli of vector bundles on curves}
\author{Pieter Belmans \\ Swarnava Mukhopadhyay}

\begin{document}
\maketitle

%\gitfootnote{commit: \texttt{\gitAbbrevHash}\hfil date: \texttt{\gitAuthorIsoDate}\hfil \texttt{\gitReferences}}

\begin{abstract}
  We show that the Poincar\'e bundle gives a fully faithful embedding from the derived category of a curve of sufficiently high genus into the derived category of its moduli space of bundles of rank~$r$ with fixed determinant of degree~1. Moreover we show that a twist of the embedding, together with 2 exceptional line bundles, gives the start of a semi-orthogonal decomposition. This generalises results of Narasimhan and Fonarev--Kuznetsov, who embedded the derived category of a single copy of the curve, for rank~2.
\end{abstract}

%\tableofcontents

\section{Introduction}
\label{section:introduction}
Let~$C$ be a smooth projective curve of genus~$g\geq 2$ and let~$r\geq 2$ and~$d$ be integers such that~$\gcd(r,d)=1$. Then the moduli space~$\moduli_C(r,\mathcal{L})$ of stable vector bundles of rank~$r$ and fixed determinant~$\mathcal{L}$ of degree~$d$ is a smooth and projective Fano variety, such that~$\Pic\moduli_C(r,\mathcal{L})\cong\mathbb{Z}\Theta$. Moreover there exists a universal vector bundle~$\mathcal{W}$ (called the \emph{Poincar\'e bundle}) on~$C\times\moduli_C(r,\mathcal{L})$. We can use~$\mathcal{W}$ to construct a Fourier--Mukai functor~$\Phi_{\mathcal{W}}$ from the derived category of the curve to the derived category of the moduli of vector bundles.

In \cite{MR3713871} Narasimhan showed that the Fourier--Mukai functor~$\Phi_{\mathcal{W}}$ is \emph{fully faithful}, if~$r=2$, $d=1$, $g\geq 4$ and~$\mathcal{W}$ is suitably normalized (see equation~\eqref{equation:normalisation}). His proof uses the Hecke correspondence to check the Bondal--Orlov criterion. Independently, in \cite{MR3764066} Fonarev--Kuznetsov showed the fully faithfulness for~$r=2$, $d=1$ and~$g\geq 2$ when~$C$ is a generic curve. Their proof involves an explicit model by Desale--Ramanan \cite{MR0429897} of~$\moduli_C(2,\mathcal{L})$ when~$C$ is hyperelliptic, and checking the Bondal--Orlov criterion in this special case using the Borel--Weil--Bott theorem.

The first result in this article is a generalisation of these fully faithfulness results: the rank~$r$ is now arbitrary, the degree~$d$ is~1, and~$g\geq g_0$, where~$g_0$ is the smallest integer such that
\begin{equation}
  \label{equation:genus-condition}
  2rg_0-2(r+g_0)\geq r^2-1.
\end{equation}
In particular when~$r=2$, we get that~$g_0$ is~$4$. The vanishing results in \cref{theorem:fully-faithful-vanishing} together with the Bondal--Orlov criterion then imply the following.
\renewcommand\thetheorem{\Alph{theorem}}
\begin{theorem}
  \label{theorem:fully-faithful}
	The Fourier--Mukai transform~$\Phi_{\mathcal{W}}$ gives a fully faithful embedding
  \begin{equation}
    \derived^\bounded(C) \rightarrow \derived^\bounded(\moduli_C(r,\mathcal{L}))
  \end{equation}
  for any smooth curve~$C$ of genus~$g\geq g_0$.
\end{theorem}

In \cite[remark~4]{MR3713871} Narasimhan explains that the results in op.~cit.~also show that
\begin{equation}
  \label{equation:start-of-sod}
  \Theta^\vee,\mathcal{O}_{\moduli_C(2,\mathcal{L})},\Phi_{\mathcal{W}}(\derived^\bounded(C))
\end{equation}
is the start of a semiorthogonal decomposition of~$\derived^\bounded(\moduli_C(r,\mathcal{L}))$. The second result in this article is a generalisation and more interestingly an extension of this result to higher rank: we exhibit a \emph{second} copy of the derived category of the curve. To do this we twist the functor~$\Phi_{\mathcal{W}}$ by~$\Theta^\vee$, or equivalently we consider the Fourier--Mukai functor associated to~$\mathcal{W}\otimes p_2^*(\Theta^\vee)$.

\begin{theorem}
	\label{theorem:sod}
  Let~$C$ be any smooth curve of genus~$g\geq g_0$, then there exists a semiorthogonal decomoposition of the form
  \begin{equation}
    \label{equation:sod}
    \derived^{\bounded}(\moduli_C(r,\mathcal{L}))
    =
    \left\langle
      \Theta^\vee,
      \Phi_{\mathcal{W}}(\derived^{\bounded}(C))\otimes\Theta^\vee,
      \mathcal{O}_{\moduli_C(r,\mathcal{L})},
      \Phi_{\mathcal{W}}(\derived^{\bounded}(C)),
      \mathcal{A}
    \right\rangle,
  \end{equation}
	where~$\mathcal{A}$ is the left-orthogonal complement to the admissible subcategory generated by the~2~exceptional objects and the~2~copies of~$\derived^\bounded(C)$.
\end{theorem}

\setcounter{theorem}{0}
\renewcommand\thetheorem{\arabic{theorem}}

\begin{remark}
  This result is already new in the case of~$r=2$. In a work in progress joint with Sergey Galkin \cite{belmans-galkin-mukhopadhyay} we are studying~$\derived^\bounded(\moduli_C(2,\mathcal{L}))$ from the point of view of mirror symmetry for Fano varieties:
  \begin{enumerate}
    \item quantum cohomology can be used to give expectations on natural semiorthogonal decompositions;
    \item there are various conjectures \cite{MR3536989} regarding the eigenvalues of the quantum multiplication~$\mathrm{c}_1(X)*-$ that can be checked for~$\moduli_C(2,\mathcal{L})$.
  \end{enumerate}
  The suggested decomposition into indecomposable pieces (recall that~$\derived^\bounded(C)$ is indecomposable by \cite{MR2838062}) will involve symmetric powers~$\Sym^i C$ for~$i\leq g-1$. For~$r\geq 3$ the picture becomes more complicated, and it is currently unclear what the systematic description could be.
\end{remark}

\paragraph{The Bondal--Orlov criterion}
To check fully faithfulness of the Fourier--Mukai functor~$\Phi_{\mathcal{W}}$, we will use following criterion due to Bondal--Orlov \cite[theorem~1.1]{alg-geom/9506012}. We denote the skyscraper at a point~$x$ by~$k(x)$.

\begin{proposition}[Bondal--Orlov]
  \label{proposition:bondal-orlov}
  Let~$X$ and~$Y$ be smooth projective varieties. Let~$\mathcal{E}$ be an object in~$\derived^\bounded(X\times Y)$. Then~$\Phi_{\mathcal{E}}$ is fully faithful if and only if
  \begin{enumerate}
    \item\label{enumerate:bondal-orlov-1} $\Hom_Y(\Phi_{\mathcal{E}}(k(x)),\Phi_{\mathcal{E}}(k(x)))\cong k$ for all~$x\in X$;
    \item\label{enumerate:bondal-orlov-2} $\Hom_Y(\Phi_{\mathcal{E}}(k(x)),\Phi_{\mathcal{E}}(k(x))[i])\cong0$ for all~$x\in X$ and~$i\notin[0,\dim X]$;
    \item\label{enumerate:bondal-orlov-3} $\Hom_Y(\Phi_{\mathcal{E}}(k(x)),\Phi_{\mathcal{E}}(k(y))[i])\cong0$ for all~$x,y\in X$ such that~$x\neq y$ and~$i\in\mathbb{Z}$.
  \end{enumerate}
\end{proposition}

\paragraph{Vanishing results}
To apply the Bondal--Orlov criterion in the proof of \cref{theorem:fully-faithful} we need the following generalisation of \cite[parts (1), (2) and (3) of theorem~1.2]{MR3713871}. For any point~$z\in C$, we denote by~$\mathcal{W}_{z}$ the restriction of~$\mathcal{W}$ to~$\{z\}\times C$.
\begin{theorem}
  \label{theorem:fully-faithful-vanishing}
  Let~$C$ be a smooth projective curve of genus~$g\geq g_0$, and~$\mathcal{L}$ a line bundle of degree~$1$. Let~$\mathcal{W}$ be the normalised Poincar\'e bundle on~$C\times\moduli_C(r,\mathcal{L})$. Then
  \begin{enumerate}
    \item\label{enumerate:fully-faithful-vanishing-1} $\HH^0(\moduli_C(r,\mathcal{L}),\mathcal{W}_x\otimes\mathcal{W}_x^\vee)\cong k$ for all~$x\in C$;
    \item\label{enumerate:fully-faithful-vanishing-2} $\HH^i(\moduli_C(r,\mathcal{L}),\mathcal{W}_x\otimes\mathcal{W}_x^\vee)\cong 0$ for all~$x\in C$ and~$i\geq 2$;
    \item\label{enumerate:fully-faithful-vanishing-3} $\HH^i(\moduli_C(r,\mathcal{L}),\mathcal{W}_x\otimes\mathcal{W}_y^\vee)\cong 0$ for all~$x,y\in C$ such that~$x\neq y$ and~$i\geq 0$.
  \end{enumerate}
\end{theorem}

For the proof of the semiorthogonal decomposition of \cref{theorem:sod} we need another set of vanishing results, which generalise \cite[parts (4) and (5) of theorem~1.2]{MR3713871}, together with a new vanishing result.
\begin{theorem}
  \label{theorem:sod-vanishing}
  Let~$C$ be a smooth projective curve of genus~$g\geq g_0$, and~$\mathcal{L}$ a line bundle of degree~$1$. Let~$\mathcal{W}$ be the normalised Poincar\'e bundle on~$C\times\moduli_C(r,\mathcal{L})$. Then
	\begin{enumerate}
    \item[1.] $\HH^i(\moduli_C(r,\mathcal{L}),\mathcal{W}_x^\vee)=0$ for all $x\in C$ and all $i\geq 0$.
		\item[2.] $\HH^i(\moduli_C(r,\mathcal{L}),\mathcal{W}_x^\vee\otimes \Theta^\vee)=0$ for all $x\in C$ and all $i\geq 0$.
		\item[2'.] $\HH^i(\moduli_C(r,\mathcal{L}),\mathcal{W}_x\otimes \Theta^\vee)=0$ for all $i\geq 0$.
		\item[3.] $\HH^i(\moduli_C(r,\mathcal{L}), \mathcal{W}_x^\vee\otimes \mathcal{W}_y\otimes \Theta^\vee)=0$ for any $x,y \in C$ and all $i\geq 0$.
	\end{enumerate}
\end{theorem}

\begin{remark}
  % need to hardcode the numbers here unfortunately
  If~$r=2$ then parts (1) and (2) are related via Serre duality, but in higher rank this is no longer the case. In arbitrary rank, (2) and (2') are related via Serre duality.
\end{remark}

\paragraph{Structure of the paper}
We prove parts (\ref{enumerate:fully-faithful-vanishing-1}) and (\ref{enumerate:fully-faithful-vanishing-2}) of \cref{theorem:fully-faithful-vanishing} in \cref{section:generalising-narasimhan-ramanan}. The proof of part (\ref{enumerate:fully-faithful-vanishing-3}) is more technical, and occupies \cref{section:high-degree-vanishing,section:low-degree-vanishing}. The structure of the proof is as in \cite{MR3713871}, but we highlight the complications arising in the higher rank case, and why one only gets a proof in the case where~$d=1$.

The proof of parts (1) and (2) (resp.~(2')) of \cref{theorem:sod-vanishing} is done in \cref{section:sod-vanishing-1-2}, where we prove a more general vanishing statement for exterior powers of~$\mathcal{W}_x^\vee$, along the same lines as \cref{section:high-degree-vanishing,section:low-degree-vanishing}. Finally, the proof of part (3) occupies \cref{section:sod-vanishing-3}.

In \cref{section:remarks} we explain how \cref{theorem:fully-faithful,theorem:sod} follows from \cref{theorem:fully-faithful-vanishing,theorem:sod-vanishing}, and give some concluding remarks.

\paragraph{Acknowledgements} Both authors were supported by the Max Planck Institute for Mathematics in Bonn. They thank Patrick Brosnan and Sergey Galkin for interesting conversations.

\section{Preliminaries and notation}
\label{section:preliminaries}
Let~$k$ be an algebraically closed field of characteristic~0. Throughout this article we will take~$C$ a smooth projective curve over~$k$ of genus~$g\geq 2$. Associated to a rank~$r$ and degree~$d$ such that~$\gcd(r,d)=1$ there exists a smooth projective moduli space~$\moduli_C(r,d)$ of dimension~$r^2(g-1)+1$, with a determinant morphism to~$\Pic^d(C)$. Fixing a line bundle~$\mathcal{L}$ of degree~$d$ on~$C$ we can consider the fibre over the point~$[\mathcal{L}]$ in~$\Pic^d(C)$, which will be denoted~$\moduli_C(r,\mathcal{L})$.

Now~$\moduli_C(r,\mathcal{L})$ has~$\Pic(\moduli_C(r,\mathcal{L}))\cong\mathbb{Z}\Theta$ with~$\Theta$ the ample generator of the Picard group, such that
\begin{equation}
  \omega_{\moduli_C(r,\mathcal{L})}\cong\Theta^{\otimes-2}.
\end{equation}
In particular, $\moduli_C(r,\mathcal{L})$ is a~$(r^2-1)(g-1)$\dash dimensional Fano variety of index~2. We will also use the appropriate analogues of these results when~$\gcd(r,d)\geq2$. For a proof one is referred to \cite{MR999313}.

As~$\moduli_C(r,\mathcal{L})$ is a \emph{fine} moduli space, there exists a universal family~$\mathcal{W}$ on~$C\times\moduli_C(r,\mathcal{L})$, the \emph{Poincar\'e bundle}. This universal family is unique if we normalise it as in \cite[remark~2.9]{MR0325615}: if we denote~$\ell\geq 0$ the minimal integer such that~$\ell d\equiv 1\bmod r$ then we can assume
\begin{equation}
  \label{equation:normalisation}
  \cc_1(\mathcal{W}_x)\cong\Theta^{\otimes\ell},
\end{equation}
where~$\mathcal{W}_x=\mathcal{W}_{\{x\}\times\moduli_C(r,\mathcal{L})}$.

\begin{remark}
  The dependence on~$\ell$ is the reason why we have to restrict to~$d\equiv 1\bmod r$. Remark that by Serre duality and tensoring with~$\mathcal{O}_C(x)$ we can always assume~$d\in[0,\ldots,r/2]$, as one can identify the moduli spaces for different~$d$, so we will use~$d=1$. We expect that for other residue classes the result still holds.
\end{remark}

Using~$\mathcal{W}$ and the projections
\begin{equation}
  \begin{tikzcd}
    & C\times\moduli_C(r,\mathcal{L}) \arrow[ld, swap, "p_1"] \arrow[rd, "p_2"] \\
    C & & \moduli_C(r,\mathcal{L})
  \end{tikzcd}
\end{equation}
we can then construct the Fourier--Mukai functor
\begin{equation}
  \Phi_{\mathcal{W}}=\RRR p_{2,*}(p_1^*(-)\otimes\mathcal{W})\colon\derived^\bounded(C)\to\derived^\bounded(\moduli_C(r,\mathcal{L})).
\end{equation}

\section{Generalisation of a result by Narasimhan--Ramanan}
\label{section:generalising-narasimhan-ramanan}
In this section we prove parts (\ref{enumerate:fully-faithful-vanishing-1}) and (\ref{enumerate:fully-faithful-vanishing-2}) of \cref{theorem:fully-faithful-vanishing}. At this point it is not required that~$d=1$, it is enough that~$\gcd(r,d)=1$.

For (\ref{enumerate:fully-faithful-vanishing-1}) it suffices to observe that the result cited in \cite{MR3713871} (i.e.~\cite[theorem~2(b)]{MR0384797}) is already valid for all ranks. However for the proof of (\ref{enumerate:fully-faithful-vanishing-2}) \cite[proposition~7.7]{MR0384797} is used, which is only given for~$g=2$. We will show that the expectation expressed in \cite[remark~7.2]{MR0384797} is correct, and that the vanishing of cohomology of~$\adjoint_x\mathcal{W}$ (i.e.~the restriction to~$\{x\}\times\moduli_C(r,\mathcal{L})$ of the traceless endomorphisms of~$\mathcal{W}$) is indeed valid in arbitrary rank, using some new information on Hodge numbers which was not available when op.~cit.~was published.

\begin{proposition}
  \label{proposition:narasimhan-ramanan-vanishing}
  Let~$C$ be a smooth projective curve of genus~$g\geq 2$, and~$\mathcal{L}$ a line bundle of degree~$d$ on~$C$ such that~$\gcd(r,d)=1$. Let~$\mathcal{W}$ be the normalised Poincar\'e bundle on~$C\times\moduli_C(r,\mathcal{L})$, where~$\gcd(r,d)=1$. For all~$x\in C$ and~$i\geq 2$ we have that
  \begin{equation}
    \label{equation:adjoint-vanishing}
    \HH^i(\moduli_C(r,\mathcal{L}),\adjoint_x\mathcal{W})=0
  \end{equation}
\end{proposition}

The proof of \cite[proposition~7.7]{MR0384797} can be generalised to arbitrary rank, provided one has control over the Hodge numbers~$\hh^{1,i}(\moduli_C(r,\mathcal{L}))$. In order to do this, we will use the closed formula for the Hodge--Poincar\'e polynomial as obtained in del~Ba\~no \cite[corollary~5.1]{MR1817504}. Recall that the Hodge--Poincar\'e polynomial~$\HP(X,x,y)$ of a smooth projective variety~$X$ is given by~$\sum_{p,q\geq 0}\hh^{p,q}(X)x^py^q$.

The formula in loc.~cit.~gives the Hodge--Poincar\'e polynomial of~$\moduli_C(r,d)$, and fixing the determinant changes the Hodge--Poincar\'e polynomial by removing a factor~$(1+x)^g(1+y)^g$ arising from the Jacobian of~$C$\footnote{There is a minor typo in the second summation in the exponent of the last factor in \cite[corollary~5.1]{MR1817504}.}. Taking these observations into account, and denoting~$\langle\alpha\rangle$ the decimal part of a real number~$\alpha$, we have
\begin{equation}
  \label{equation:del-bano}
  \begin{aligned}
    \HP(\moduli_C(r,\mathcal{L}),x,y)
    &=
    \sum_{r_1+\ldots+r_\ell=r}
    (-1)^{\ell-1}\frac{\left( (1+x)^g(1+y)^g \right)^{\ell-1}}{(1-xy)^{\ell-1}} \\
    &\qquad\qquad\qquad\cdot\prod_{j=1}^\ell\prod_{i=1}^{r_j-1}\frac{(1+x^iy^{i+1})^g(1+x^{i+1}y^i)^g}{(1-(xy)^i)(1-(xy)^{i+1})} \\
    &\qquad\qquad\qquad\cdot\prod_{j=1}^{\ell-1}\frac{1}{1-(xy)^{r_j+r_{j+1}}} \\
    &\qquad\qquad\qquad\cdot(xy)^{\sum_{i<j}r_ir_j(g-1)+\sum_{i=1}^{\ell-1}(r_i+r_{i+1})\langle -(r_1+\ldots+r_i)d/r\rangle}
  \end{aligned}
\end{equation}
where we sum over all compositions of~$r$.

From this we can read off the following dimensions.
\begin{lemma}
  \label{lemma:hodge-numbers}
  We have that
  \begin{enumerate}
    \item $\hh^{0,1}(\moduli_C(r,\mathcal{L}),)=0$,
    \item $\hh^{1,1}(\moduli_C(r,\mathcal{L}),)=1$,
    \item $\hh^{2,1}(\moduli_C(r,\mathcal{L}),)=g$,
    \item $\hh^{i,1}(\moduli_C(r,\mathcal{L}),)=0$ for all~$i\geq 3$.
  \end{enumerate}
\end{lemma}

\begin{proof}
  We observe that the only composition~$r=r_1+\ldots+r_\ell$ contributing to the coefficient of~$x^iy$ is the composition with~$\ell=1$. Indeed, developing all the denominators as a power series in~$xy$ and multiplying them, we see that factor on the fourth line of \eqref{equation:del-bano} is zero for~$\ell=1$ and bounded below by~2 for~$\ell\geq 2$. Hence for the equalities in the lemma, it suffices to understand the contribution with~$\ell=1$.

  In this case the only non-trivial factor in \eqref{equation:del-bano} is the second one. It now suffices to observe that there will not be a monomial~$y$, that~$xy$ appears with coefficient~$1$ by the factor~$(1-xy)$ in the denominator for~$i=1$, that~$x^2y$ appears with coefficient~$g$ by the factor~$(1+x^2y)^g$ for~$i=2$ and that there cannot be any monomials of the form~$x^iy$ for~$i\geq 3$.
\end{proof}

\begin{proof}[Proof of part (\ref{enumerate:fully-faithful-vanishing-2}) of \cref{theorem:fully-faithful-vanishing}]
  It suffices to combine \cref{proposition:narasimhan-ramanan-vanishing} with the defining short exact sequence
  \begin{equation}
    0\to\adjoint_x\mathcal{W}\to\sheafEnd(\mathcal{W}_x)\cong\mathcal{W}_x\otimes\mathcal{W}_x^\vee\to\mathcal{O}_{\moduli_C(r,\mathcal{L})}\to 0,
  \end{equation}
  and the vanishing of the higher cohomology of the structure sheaf as~$\moduli_C(r,\mathcal{L})$ is Fano.
\end{proof}

\begin{remark}
  One can actually prove more. Recall that the \emph{level} of a Hodge structure is defined as the maximum value of~$|p-q|$ as~$(p,q)$ ranges over~$\hh^{p,q}\neq 0$. In \cite{MR1895918} it was shown that for~$\moduli_C(2,\mathcal{L})$ the level of~$\HH^i(\moduli_C(2,\mathcal{L}),\mathbb{Q})$ was bounded above by~$\lfloor i/3\rfloor$. For~$r\geq 3$ the same is true, which immediately implies~\cref{lemma:hodge-numbers}. For~$i\gg 0$ this bound can even be improved, but we haven't found a nice closed formula for it.
\end{remark}

\section{Determinant of cohomology}
\label{section:determinant-of-cohomology}
One of the main ingredients in the proof of part (\ref{enumerate:fully-faithful-vanishing-3}) of \cref{theorem:fully-faithful-vanishing} is an explicit description of the determinant of cohomology. The following proposition is a generalisation of \cite[proposition~3.1]{MR3713871}, taking the extra complication for arbitrary~$r$ and~$d$ such that~$\gcd(r,d)=1$ into account.

\begin{proposition}
  \label{proposition:determinant-of-cohomology}
  Let~$C$ be a smooth projective curve of genus~$g\geq 2$, and~$\mathcal{L}$ a line bundle of degree~$d$ on~$C$ such that~$\gcd(r,d)=1$. Let~$\mathcal{W}$ be the normalised Poincar\'e bundle on~$C\times\moduli_C(r,\mathcal{L})$, where~$\gcd(r,d)=1$. Then there exists an isomorphism
  \begin{equation}
    \det\left( \RRR p_{2,*}(\mathcal{W})^\vee \right)^\vee\cong\mathcal{L}^{\otimes(1-\ell d)/r+\ell(1-g)}
  \end{equation}
  where~$\ell\geq 0$ is minimal such that~$\ell d\equiv 1\bmod r$.
\end{proposition}

\begin{proof}
  Consider the moduli space~$\moduli_C(r,\mathcal{L}^\vee)$. The family of vector bundles~$\mathcal{W}^\vee$ on~$C\times\moduli_C(r,\mathcal{L})$ gives an isomorphism
  \begin{equation}
    \phi\colon\moduli_C(r,\mathcal{L})\overset{\cong}{\to}\moduli_C(r,\mathcal{L}^\vee).
  \end{equation}
  Let us denote the ample generators of~$\Pic(\moduli_C(r,\mathcal{L}))$ (resp.~$\Pic(\moduli_C(r,\mathcal{L}^\vee))$) by~$\Theta$ (resp.~$\Theta'$). By \cite[proposition~2.1]{MR3713871} we get that
  \begin{equation}
    \Theta\cong\phi^*(\Theta')\cong\left( \det(\RRR p_{2,*}\mathcal{W}^\vee)^\vee \right)^{\otimes r}\otimes\det(\mathcal{W}_x^\vee)^{\otimes -d+r(1-g)}.
  \end{equation}
  Because~$\mathcal{W}$ is normalised we get~$\det\mathcal{W}_x\cong\Theta^\ell$ as explained in \cref{section:preliminaries}, where~$\ell$ is the minimal non-negative integer such that~$\ell d\equiv 1\bmod r$. Hence we obtain
  \begin{equation}
    \Theta\cong\left( \det(\RRR p_{2,*}\mathcal{W}^\vee)^\vee \right)^{\otimes r}\otimes\Theta^{\otimes\ell d+\ell r(g-1)}.
  \end{equation}
  This proves the proposition.
\end{proof}

In \cref{appendix:alternative-proof} we give an alternative proof of this result using the Grothendieck--Riemann--Roch formula.

\section{Cohomology vanishing for \texorpdfstring{$\mathcal{W}_x\otimes\mathcal{W}_y^\vee$}{Wx tensor Wy dual} in high degrees}
\label{section:high-degree-vanishing}
From this point on we will impose the condition that~$d=1$. As mentioned before, we expect the result is still valid for any degree coprime to the rank, but leave this for future work.

We will split the proof of the vanishing as stated in part (\ref{enumerate:fully-faithful-vanishing-3}) of \cref{theorem:fully-faithful-vanishing} into two parts: first we show it for~$i\geq r^2$, and then we deal with the vanishing in low degrees in \cref{section:low-degree-vanishing}. For both parts we use the \emph{Hecke correspondence}, which we will recall now.

\paragraph{Hecke correspondence}
Let~$x\in C$ be a closed point. If we restrict the Poincar\'e bundle~$\mathcal{W}$ to~$\{x\}\times\moduli_C(r,\mathcal{L})$, we can consider the projective bundle~$\pi\colon\mathbb{P}(\mathcal{W}_x)\to\moduli_C(r,\mathcal{L})$. We will denote the projective bundle by~$Q(\mathcal{W},x)$.

A point~$q$ on~$Q(\mathcal{W},x)$ corresponds to the vector bundle~$\mathcal{W}_{\pi(q)}$ together with a morphism~$\mathcal{W}_{\pi(q)}\twoheadrightarrow k(x)$. We can take the dual of the kernel, which is a vector bundle of rank~$r$ and determinant~$\mathcal{L}^\vee\otimes\mathcal{O}_C(x)$ of degree~$1-d$. The variety~$Q(\mathcal{W},x)$ can be seen to parametrise a family of such bundles. However in general the family of vector bundles may not be semistable.

Consider the short exact sequence
\begin{equation}
  0\rightarrow \mathcal{E}_1\rightarrow \mathcal{E} \rightarrow k(x)\rightarrow 0,
\end{equation}
on a curve~$C$, where~$\mathcal{E}_1$ and~$\mathcal{E}$ are vector bundles of rank~$r$ and~$\deg(\mathcal{E})=1$. Then we have the following:
\begin{lemma}
  If~$\mathcal{E}$ is stable, then~$\mathcal{E}_1$ is semistable.
\end{lemma}

\begin{proof}
  First observe that the slope~$\mu(\mathcal{E}_1)=0$. Suppose that~$\mathcal{E}$ is not semistable. Let~$\mathcal{F}'$ be a subbundle of~$\mathcal{E}_1$ such that~$\mu(\mathcal{F}')>\mu(\mathcal{E}_1)=0$. Now since~$\mathcal{E}$ is stable, we have
  \begin{equation}
    0 < \mu(F) <\mu(E)=\frac{1}{r}.
  \end{equation}
  But this is impossible since $0<r'<r$ and $\deg(\mathcal{F}')>0$.
\end{proof}

The above lemma shows that Hecke transforms for degree one bundles preserves semistability. Thus we get a diagram
\begin{equation}
  \label{equation:hecke}
  \begin{tikzcd}
    & Q(\mathcal{W},x)\coloneqq\mathbb{P}(\mathcal{W}_x) \arrow[ld, swap, "\psi"] \arrow[rd, "\pi"] \\
    \moduli_C(r,\mathcal{L}^\vee\otimes\mathcal{O}_C(x)) & & \moduli_C(r,\mathcal{L}).
  \end{tikzcd}
\end{equation}

\begin{remark}
  When~$d\neq 1$ we don't have a well-defined morphism~$\psi$ in \eqref{equation:hecke}, there is only a rational morphism. It is conceivable that using parabolic bundles it is possible to resolve the indeterminacy, and continue the proof in this way. We leave this for future work.
\end{remark}

With the above considerations, the statement and proof of \cite[proposition~3.3]{MR3713871} generalises to the following.

\begin{lemma}
  \label{lemma:pullback}
  Let~$\Theta'$ be the ample generator of~$\Pic(\moduli_C(r,\mathcal{L}^\vee\otimes\mathcal{O}_C(x))$. Then
  \begin{equation}
    \label{equation:pullback-ample-generator}
    \psi^*(\Theta')\cong\mathcal{O}_{Q(\mathcal{W},x)}(1)
  \end{equation}
\end{lemma}

The following proposition is a generalisation of \cite[proposition~3.2]{MR3713871}.
\begin{proposition}
  \label{proposition:ample-vector-bundles}
  Let~$C$ be a smooth projective curve of genus~$g\geq 2$, and~$\mathcal{L}$ a line bundle of degree~$1$ on~$C$. Let~$\mathcal{W}$ be the normalised Poincar\'e bundle on~$C\times\moduli_C(r,\mathcal{L})$. For all~$x,y\in C$ such that~$x\neq y$ the vector bundle~$\mathcal{W}_x\otimes\mathcal{W}_y^\vee\otimes\omega_{\moduli_C(r,\mathcal{L})}^\vee$ is ample.
\end{proposition}

%\begin{proof}
%  By \cref{lemma:pullback} we have that~$\psi^*(\Theta')=\mathcal{O}_{Q(\mathcal{W},x)}(1)$. Since~$\Theta'$ is ample this implies that~$\mathcal{O}_{Q(\mathcal{W},x)}(1)$ is nef. By the normalisation of~$\mathcal{W}_x$, we get~$\det(\mathcal{W}_x)=\Theta$. Now given the isomorphism
%  \begin{equation}
%    \mathcal{W}_x^\vee\cong\bigwedge\nolimits^{r-1}\mathcal{W}_x\otimes\Theta^\vee
%  \end{equation}
%  and the fact that exterior product preserves the property of being nef, the proof of \cite[proposition~3.2]{MR3713871} applies here too.
%\end{proof}

\begin{proof}
  By \cref{lemma:pullback} we get that~$\mathcal{O}_{Q(\mathcal{W},x)}(1)$ is nef, hence~$\mathcal{W}_x$ is nef. By \cite[proposition~6.2.12(iv)]{MR2095472} the exterior power~$\bigwedge^{r-1}\mathcal{W}_x$ is again nef, and so~$\bigwedge^i\mathcal{W}_x\otimes\Theta$ is ample. Using the ampleness of
  \begin{equation}
    \begin{aligned}
      \bigwedge^{r-1}\mathcal{W}_x\otimes\Theta
      &\cong\bigwedge^{r-1}\mathcal{W}_x\otimes\Theta^\vee\otimes\mathcal{L}^{\otimes 2} \\
      &\cong\bigwedge^{r-1}\mathcal{W}_x\otimes\det\mathcal{W}_x^\vee\otimes\Theta^{\otimes 2} \\
      &\cong\mathcal{W}_x^\vee\otimes\omega_{\moduli_C(r,\mathcal{L})}^\vee
    \end{aligned}
  \end{equation}
  we can conclude that
  \begin{equation}
    \begin{aligned}
      \mathcal{W}_x\otimes\mathcal{W}_y^\vee\otimes\omega_{\moduli_C(r,\mathcal{L})}^\vee
      &\cong\mathcal{W}_x\otimes\Theta\otimes\mathcal{W}_y^\vee\otimes\Theta \\
      &\cong\mathcal{W}_x\otimes\Theta\otimes\bigwedge^{r-1}\mathcal{W}_y
    \end{aligned}
  \end{equation}
  is again the tensor product of a nef bundle with an ample bundle, so is ample itself.
\end{proof}

From this we get the vanishing of~$\mathcal{W}_x\otimes\mathcal{W}_y^\vee$ in high degree.

\begin{corollary}
  \label{corollary:le-potier-applied}
  We have that for~$i\geq r^2$.
  \begin{equation}
    \HH^i(\moduli_C(r,\mathcal{L}),\mathcal{W}_x\otimes\mathcal{W}_y^\vee)=0
  \end{equation}
\end{corollary}

\begin{proof}
  It suffices to apply Le Potier vanishing \cite[theorem~7.3.5]{MR2095472} to the ample vector bundle~$\mathcal{W}_x\otimes\mathcal{W}_y^\vee\otimes\omega_{\moduli_C(r,\mathcal{L})}^\vee$ of rank~$r^2$.
\end{proof}

\section{Cohomology vanishing for \texorpdfstring{$\mathcal{W}_x\otimes\mathcal{W}_y^\vee$}{Wx tensor Wy dual} in low degrees}
\label{section:low-degree-vanishing}
To finish the proof of part~(\ref{enumerate:fully-faithful-vanishing-3}) \cref{theorem:fully-faithful-vanishing} we need to show vanishing in degrees~$i\leq r^2-1$. For this we will use the morphism~$\psi$ from \eqref{equation:hecke}.

Because~$\deg\mathcal{L}^\vee\otimes\mathcal{O}_C(x)=0$ is not coprime with~$r$, there is a dense open subset
\begin{equation}
  \moduli_C^{\mathrm{s}}(r,\mathcal{L}^\vee\otimes\mathcal{O}_C(x))\subsetneq\moduli_C(r,\mathcal{L}^\vee\otimes\mathcal{O}_C(x))
\end{equation}
of stable vector bundles, the complement being the strictly semistable locus.

\begin{remark}
  \label{remark:01-stable}
  If we take a stable bundle~$\mathcal{V} \in \moduli_C(r,\mathcal{L}^\vee\otimes \mathcal{O}_C(x))$, then by \cite[remark~5.2(v) and lemma~5.6(i)]{MR541029} the Hecke transform of~$\mathcal{V}$ is also stable.
\end{remark}

Using \cref{remark:01-stable}, the proof of \cite[proposition~3.4]{MR3713871} also gives a proof of the following generalisation.
\begin{proposition}
  \label{proposition:restriction-to-fibre}
  The restriction of~$\psi$ to~$\moduli_C^{\mathrm{s}}(r,\mathcal{L}^\vee\otimes\mathcal{O}_C(x))$ is a~$\mathbb{P}^{r-1}$\dash fibration. Moreover, for every point~$y\in C\setminus\{x\}$ the restriction of~$\pi^*(\mathcal{W}_y)$ to a fibre of~$\psi$ in the stable locus is isomorphic to~$\mathcal{O}_{\mathbb{P}^{r-1}}(1)^{\oplus r}$.
\end{proposition}
For completeness' sake, we give the proof.
\begin{proof}
  Consider a point~$m\in\moduli_C^{\mathrm{s}}(r,\mathcal{L}^\vee\otimes\mathcal{O}_C(x))$, which corresponds to a stable vector bundle~$\mathcal{V}$ on~$C$ of rank~$r$ and determinant~$\mathcal{L}^\vee\otimes\mathcal{O}_C(x)$. We wish to show that
  \begin{equation}
    \label{equation:desired-isomorphism}
    \pi^*(\mathcal{W}_y)|_{\psi^{-1}(m)}\cong\mathcal{O}_{\mathbb{P}^{r-1}}(1)^{\oplus r}.
  \end{equation}
  Let us denote~$V_x=\mathcal{V}_x\otimes k(x)$ the fibre of~$\mathcal{V}$ at~$x$. Then as in \cite[\S4]{MR0384797} we obtain a family of vector bundles~$K(\mathcal{V})$ on~$C$ parametrised by~$\mathbb{P}(V_x)\cong\mathbb{P}^{r-1}$, where~$T=\Spec k$ is just a point. It should be mentioned that we will also use results from \cite[\S5]{MR541029}, and it is important to highlight remark~5.7 of op.~cit., which explains how the results in \S5 of op.~cit.~are to be interpreted in the context of \cite[\S4]{MR0384797}, in particular we have~$x\in C$ fixed.

  As~$K(\mathcal{V})$ is a family of stable vector bundles of rank~$r$ and determinant~$\mathcal{L}$, the universal property of~$\moduli_C(r,\mathcal{L})$ gives us a morphism
  \begin{equation}
    f\colon\mathbb{P}(V_x)\to\moduli_C(r,\mathcal{L})
  \end{equation}
  which is a closed immersion by~\cite[lemma~5.9]{MR541029}, where in the setting of loc.~cit. we restricted the morphism to the point~$x\in C$. The same universal property gives us an isomorphism
  \begin{equation}
    \label{equation:universal-property}
    (f\times\identity_C)^*(\mathcal{W})\cong K(\mathcal{V})\otimes g^*(\xi)
  \end{equation}
  on~$\mathbb{P}(V_x)\times C$, for some line bundle~$\xi$ on~$\mathbb{P}(V_x)$, where~$g\colon\mathbb{P}(V_x)\times C\to\mathbb{P}(V_x)$ is the projection onto the first factor. We wish to determine the line bundle~$\xi$ more explicitly, and we do this by computing the restriction~$f^*(\Theta)$.

  First, by \cite[remark~5.7 and corollary~5.16]{MR541029} we get an isomorphism
  \begin{equation}
    \label{eqn:heckelemma}
    f^*(\omega_{\moduli_C(r,\mathcal{L})})\cong\omega_{\mathbb{P}(V_x)}^{\otimes 2},
  \end{equation}
  and as~$\moduli_C(r,\mathcal{L})$ is Fano of index~2 we get
  \begin{equation}
    f^*(\Theta)\cong\mathcal{O}_{\mathbb{P}(V_x)}(r).
  \end{equation}
  On the other hand, using \eqref{equation:universal-property} we get
  \begin{equation}
    \begin{aligned}
      f^*(\Theta)
      &\cong f^*(\det(\mathcal{W}_y)) \\
      &\cong\det(K(\mathcal{V})|_{\mathbb{P}(V_x)\times\{y\}}\otimes\xi).
    \end{aligned}
  \end{equation}
  But as~$K(\mathcal{V})|_{\mathbb{P}(V_x)\times\{y\}}\cong\mathcal{O}_{\mathbb{P}(V_x)}^{\oplus r}$ by \cite[remark~4.7]{MR0384797}, we get
  \begin{equation}
    f^*(\Theta)\cong\xi^{\otimes r}.
  \end{equation}
  Hence~$\xi\cong\mathcal{O}_{\mathbb{P}(V_x)}(1)$, and \eqref{equation:universal-property} gives
  \begin{equation}
    f^*(\mathcal{W}_y)\cong\mathcal{O}_{\mathbb{P}(V_x)}(1)^{\oplus r}.
  \end{equation}

  Finally, similar to \cite[\S5]{MR0384797} there exists a commutative diagram
  \begin{equation}
    \begin{tikzcd}
      & \mathbb{P}(\mathcal{W}_x) \arrow[d, "\pi"] \\
      \mathbb{P}(V_x) \arrow[r, hook, "f"] \arrow[ru, hook, "F"] & \moduli_C(r,\mathcal{L})
    \end{tikzcd}
  \end{equation}
  such that the image of~$F$ is identified with~$\psi^{-1}(m)$. But then \eqref{equation:desired-isomorphism} follows.
\end{proof}

%\begin{remark}
%  Since $K(\mathcal{V})_{\mathbb{P}(V_x)\times y}$ is trivial, the isomorphism of line bundles in equation \ref{eqn:heckelemma} follows from  \cite[proposition~2.1]{MR3713871} and the fact that $\Det K(\mathcal{V})=\mathcal{O}_{\mathbb{P}(V_x)}(1)$. The last equality follows from behaviour of Determinant of cohomology in exact sequence and the fact that if $E$ is a degree zero vector bundle on $X$, then on $X\times \mathbb{P}^{N}$, the $\Det p_2^*E$ is trivial.
%\end{remark}

We will need the following codimension estimate for the preimage under~$\psi$ for the strictly semistable locus. The proof is similar to that in \cite[\S 2.V.A]{MR699278}, but we will now fix the determinant.
\begin{lemma}
  \label{lemma:codimension}
  Denoting~$U\coloneqq\psi^{-1}(\moduli_C^{\mathrm{s}}(r,\mathcal{L}^\vee\otimes\mathcal{O}_C(x))$ we have that
  \begin{equation} % codim=-1 if gcd(r,1-d)=1
    \codim_{Q(\mathcal{W},x)}(Q(\mathcal{W},x)\setminus U)\geq 2(r-1)(g-1)-1
  \end{equation}
%  where~$r_0=r/\gcd(r,1-d)$.
\end{lemma}

\begin{proof}
  Let~$\mathcal{M}$ be a line bundle of degree~$m$. Recall that polystable bundles in
  \begin{equation}
    K\coloneqq\moduli_C(r,\mathcal{M})\setminus\moduli_C^{\mathrm{s}}(r,\mathcal{M})
  \end{equation}
  are direct sums of stable bundles with the same slope. If we denote~$a=\gcd(r,m)\geq 2$ and~$r_0=r/a$ we get
  \begin{equation}
    \begin{aligned}
      \dim K
      &=\max_{c=1,\ldots,\lfloor a/2\rfloor}\left( r_0^2(c^2+(a-c)^2(g-1)+2-g \right) \\
      %&=r_0^2(1+(a-1)^2)(g-1)+2-g \\
      &=(2r_0^2+r^2-2rr_0-1)(g-1)+1,
    \end{aligned}
  \end{equation}
  where we can choose the determinant for one of the summands freely.

  In the situation of \eqref{equation:hecke} we have~$d=1$, and hence~$r_0=1$. By a generalisation of the proof of \cite[lemma~5]{MR699278}, it follows that the fibers of the map~$\psi$ are at most~$(r-1)$\dash dimensional. Thus we see that the dimension of $\psi^{-1}(K)$ is at most~$(2+r^2-2r-1)(g-1)+1+(r-1)$. Hence we get that
  \begin{equation}
    \begin{aligned}
      &\codim_{Q(\mathcal{W},x)}(Q(\mathcal{W},x)\setminus U) \\
      &\qquad = (r^2-1)(g-1)+r-1 - \dim \psi^{-1}(K)\\
%      &\qquad\geq  \left( (r^2-1)(g-1)+r-1 \right)-\left( (2+r^2-2r-1)(g-1)+1+(r-1) \right) \\
      &\qquad\geq  2(r-1)(g-1)-1.
    \end{aligned}
  \end{equation}
 \end{proof}

\begin{remark}
  If~$r=2$, it is known that~$\psi$ is surjective by \cite[lemma~7.3]{MR541029} and the inequality in the statement of \cref{lemma:codimension} is an equality.
\end{remark}

In \cref{corollary:le-potier-applied} we have shown the vanishing of~$\HH^i(\moduli_C(r,\mathcal{L}),\mathcal{W}_x\otimes\mathcal{W}_y^\vee)$ for~$i\geq r^2$. We can now show vanishing for~$i\leq r^2-1$.

\begin{proof}[Proof of part (\ref{enumerate:fully-faithful-vanishing-3}) of \cref{theorem:fully-faithful-vanishing}]
  By the Leray spectral sequence for~$\psi$ and \eqref{equation:pullback-ample-generator} we get that
  \begin{equation}
    \begin{aligned}
      \HH^i(\moduli_C(r,\mathcal{L}), \mathcal{W}_x\otimes \mathcal{W}_y^\vee)
      &\cong\HH^i(Q(\mathcal{W},x),\mathcal{O}_{Q(\mathcal{W},x)}(1)\otimes\pi^*(\mathcal{W}_y^\vee)) \\
      &\cong\HH^i(Q(\mathcal{W},x),\psi^*(\Theta')\otimes\pi^*(\mathcal{W}_y^\vee)).
    \end{aligned}
  \end{equation}
  Let~$g_0$ be the smallest integer such that
  \begin{equation}
    2rg_0-2(r+g_0)\geq r^2-1,
  \end{equation}
  where we have used the codimension estimate from \cref{lemma:codimension}. Then by \cite[\S III.3]{MR0476737} we have an inclusion
  \begin{equation}
    \HH^i(Q(\mathcal{W},x),\psi^*(\Theta')\otimes\pi^*(\mathcal{W}_y^\vee))
    \hookrightarrow
    \HH^i(U,\psi^*(\Theta')\otimes\pi^*(\mathcal{W}_y^\vee))
  \end{equation}
  for~$i\leq r^2-1$.

  By \cref{proposition:restriction-to-fibre} we have that~$\psi$ restricted to~$U$ is a~$\mathbb{P}^{r-1}$\dash fibration. Applying the projection formula
  \begin{equation}
    \RR^j\psi_*\left( \psi^*(\Theta')\otimes\pi^*(\mathcal{W}_y^\vee) \right)\cong\Theta\otimes\RR^j\psi_*(\pi^*(\mathcal{W}_y^\vee))
  \end{equation}
  and using that the fibers of~$\pi^*(\mathcal{W}_y^\vee)$ restricted to~$\psi^{-1}(p)$ for~$p\in\moduli_C^{\mathrm{s}}(r,\mathcal{L}^\vee\otimes\mathcal{O}_C(x))$ are isomorphic to~$\mathcal{O}_{\mathbb{P}^{r-1}}(-1)^{\oplus r}$ we are done.
\end{proof}

\section{Cohomology vanishing for \texorpdfstring{$\mathcal{W}_x^\vee$}{Wx dual} and \texorpdfstring{$\mathcal{W}_x\otimes\Theta^\vee$}{Wx tensor Theta dual}}
\label{section:sod-vanishing-1-2}
In this section we prove parts (1), (2) and (2') of \cref{theorem:sod-vanishing}. The proof goes along the same lines as the proof of part (\ref{enumerate:fully-faithful-vanishing-3}) of \cref{theorem:fully-faithful-vanishing}, and uses ingredients from \cref{section:high-degree-vanishing,section:low-degree-vanishing}.

To prove cohomology vanishing for~$\mathcal{W}_x^\vee$ and~$\mathcal{W}_x\otimes\Theta^\vee$ we prove the following more general vanishing statement for~$\bigwedge^j\mathcal{W}_x^\vee$ with~$j=1,\ldots,r-1$. Setting~$j=1$ and~$j=r-1$ then implies the result, as
\begin{equation}
  \bigwedge^{r-1}\mathcal{W}_x^\vee\cong\mathcal{W}_x\otimes\Theta^\vee
\end{equation}
because~$\det\mathcal{W}_x\cong\Theta$.

\begin{proposition}
  For any~$i\geq 0$ and for $j=1,\ldots,r-1$ we have that
  \begin{equation}
    \HH^i(\moduli_{C}(r,\mathcal{L}),\bigwedge^j\mathcal{W}_x^\vee)=0
  \end{equation}
  whenever~$g\geq g_j$, where~$g_j$ is the smallest positive integer satisfying the condition
  \begin{equation}
    2rg_j-2(r+g_j)\geq\binom{r}{j}-1.
  \end{equation}
\end{proposition}

\begin{proof}
  The strategy of the proof is similar to the proof of part (\ref{enumerate:fully-faithful-vanishing-3}) of \cref{theorem:fully-faithful-vanishing}. First we claim that~$\bigwedge^j\mathcal{W}_x^\vee\otimes \omega_{\moduli_C(r,\mathcal{L})}$ is ample. This follows from the isomorphism
  \begin{equation}
    \bigwedge^j\mathcal{W}_x^\vee\otimes\omega_{\moduli_C(r,\mathcal{L})}\cong\bigwedge^{r-j}\mathcal{W}_x\otimes \Theta
  \end{equation}
  and the fact that~$\mathcal{W}_x$ is nef, using \cref{lemma:pullback}, as in the proof of \cref{proposition:ample-vector-bundles}

  Thus by Le Potier vanishing \cite[theorem~7.3.5]{MR2095472} we get that
  \begin{equation}
    \HH^i(\moduli_C(r,\mathcal{L}), \bigwedge^{j}\mathcal{W}^\vee_x)=0
  \end{equation}
  for~$i \geq\binom{r}{j}$, as in \cref{section:high-degree-vanishing}. Thus now we only have to prove vanishing when $i<\binom{r}{j}$.

  As in \cref{section:low-degree-vanishing} the Leray spectral sequence tells us
  \begin{equation}
    \HH^i(\moduli_{C}(r,\mathcal{L}),\bigwedge^j\mathcal{W}_x^\vee)\cong\HH^i(Q(\mathcal{W},z),\bigwedge^j\pi^*(\mathcal{W}_x^\vee),
  \end{equation}
  where~$z$ is any point on~$C\setminus\{x\}$.

  Assuming that~$g\geq g_j$ and using the codimension estimate of \cref{lemma:codimension} together with~\cite[\S III.3]{MR0476737}, we obtain an inclusion
  \begin{equation}
    \HH^i(Q(\mathcal{W},z),\bigwedge^j\pi^*(\mathcal{W}_x^\vee)\subseteq\HH^i(U,\bigwedge^{j}\pi^*(\mathcal{W}_x^\vee)
  \end{equation}
  for any~$i<\binom{r}{j}$ and~$U$ as in \cref{lemma:codimension}. Now by \cref{proposition:restriction-to-fibre}, we get that~$\psi$ restricted to~$U$ is a~$\mathbb{P}^{r-1}$\dash fibration.

  Moreover by \cref{proposition:restriction-to-fibre}, we see that the fiber of~$\pi^*(\mathcal{W}_x^\vee)$ restricted to~$\psi^{-1}(p)$ for~$p \in \moduli_{C}^{\mathrm{s}}(r,\mathcal{L}^\vee\otimes \mathcal{O}_C(z))$ is isomorphic to~$\mathcal{O}_{\mathbb{P}^{r-1}}(-1)^{\oplus r}$. In particular, the fibers of~$\bigwedge^j\pi^*(\mathcal{W}_x^\vee)$ restricted to~$\psi^{-1}(p)$ are direct sums of line bundles of degree~$-j$. This implies that~$\RR^k\psi_*(\bigwedge^j\pi^*(\mathcal{W}_x^\vee))$ is zero, hence we are done.
\end{proof}

\section{Cohomology vanishing for \texorpdfstring{$\mathcal{W}_x^\vee\otimes\mathcal{W}_y\otimes\Theta^\vee$}{Wx dual tensor Wy tensor Theta dual}}
\label{section:sod-vanishing-3}
In this section we discuss the vanishing result of part (3) of \cref{theorem:sod-vanishing} required for the proof of \cref{theorem:sod}. We will recall the notion of~$k$\dash ample vector bundles, as the role of Le Potier vanishing in the proof of \cref{theorem:fully-faithful-vanishing} will be replaced by Sommese vanishing.

\begin{definition}
  Let~$X$ be a projective variety. A line bundle~$\mathcal{L}$ is said to be \emph{$k$\dash ample} if
	\begin{enumerate}
    \item some power of~$\mathcal{L}$ is globally generated, i.e.~$\mathcal{L}$ is semi-ample;
		\item the fibers of the morphism
      \begin{equation}
        X \rightarrow \mathbb{P}(\HH^0(X,\mathcal{L}^{\otimes r})^\vee)
      \end{equation}
      have dimension at most~$k$.
	\end{enumerate}
  A vector bundle~$\mathcal{E}$ on~$X$ is said to be \emph{$k$\dash ample} if the line bundle~$\mathcal{O}_{\mathbb{P}(\mathcal{E})}(1)$ is~$k$\dash ample.
\end{definition}

The notion of~$0$\dash ampleness agrees with ampleness and furthermore~$k$\dash ampleness implies~$(k+1)$\dash ampleness.

\begin{proposition}
  \label{proposition:k-ample}
  Let~$C$ be a smooth projective curve of genus~$g\geq 2$ and~$\mathcal{L}$ a line bundle of degree~$1$. Let~$\mathcal{W}$ be the normalised Poincar\'e bundle on~$C\times\moduli_C(r,\mathcal{L})$. Then~$\mathcal{W}_x$ is~$(r-1)$\dash ample.
\end{proposition}

\begin{proof}
  By the generalisation of \cite[lemma~5]{MR699278} we know that the fibres of the morphism~$\psi$ in the Hecke correspondence of \eqref{equation:hecke} are at most~$(r-1)$\dash dimensional. Moreover, we have already (see \ref{equation:pullback-ample-generator}) shown that
  \begin{equation}
    \mathcal{O}_{Q(\mathcal{W},x)}(1)\cong\Theta',
  \end{equation}
	where~$\Theta'$ is the ample generator of the Picard group of~$\moduli_C(r,\mathcal{L}^\vee\otimes \mathcal{O}_{C(x)})$. Since~$\Theta'$ is~(0-)ample, we get by \cite[example~6.2.19(ii)]{MR2095472}) that~$\mathcal{O}_{Q(\mathcal{W},x)}(1)$ is~$(r-1)$\dash ample. In particular this implies that $\mathcal{W}_x$ is $(r-1)$-ample.
\end{proof}

We will use this to prove part (3) of \cref{theorem:sod-vanishing}.

\begin{proof}[Proof of part (3) of \cref{theorem:sod-vanishing}]
  Let~$x$ and~$y$ be any two points of~$C$, not necessarily distinct, and denote
  \begin{equation}
    \mathcal{V}\coloneqq\mathcal{W}_y\otimes \mathcal{W}^\vee_x \otimes \Theta.
  \end{equation}
  By \cref{proposition:k-ample} we know that~$\mathcal{W}_x$ and~$\mathcal{W}_y$ are~$(r-1)$\dash ample. Now by \cite[theorem~3.5]{1607.07193v1} we get that both~$\bigwedge^{r-1}\mathcal{W}_x$ and~$\mathcal{W}_y\otimes \bigwedge^{r-1} \mathcal{W}_x$ are~$(r-1)$\dash ample. By our choice of normalisation of~$\mathcal{W}$ there is a natural isomorphism
  \begin{equation}
    \mathcal{V} \cong\mathcal{W}_y\otimes \mathcal{W}^\vee_x\otimes \Theta \cong  \mathcal{W}_y\otimes\bigwedge^{r-1}\mathcal{W}_x.
  \end{equation}
  Thus, we can conclude that the vector bundle~$\mathcal{V}$ is~$(r-1)$\dash ample. It follows by Sommese vanishing \cite[proposition~1.14]{MR0466647}, that
  \begin{equation}
    \label{equation:sommese-vanishing}
    \HH^i(\moduli_C(r,\mathcal{L}),\mathcal{W}_y\otimes \mathcal{W}_x^\vee\otimes \Theta^\vee)=0
  \end{equation}
  for all~$i \geq r^2+(r-1)$.

  Since~$\omega^\vee_{\moduli_C(r,\mathcal{L})}$ is isomorphic to~$\Theta^{\otimes 2}$ we get by Serre duality that
  \begin{equation}
    \label{equation:serredualitytrick}
    \HH^i(\moduli_C(r,\mathcal{L}),\mathcal{W}_y\otimes \mathcal{W}_x^\vee\otimes \Theta^\vee)^\vee\cong\HH^{(r^2-1)(g-1)-i}(\moduli_C(r,\mathcal{L}),\mathcal{W}_x\otimes \mathcal{W}_y^\vee\otimes \Theta^\vee).
  \end{equation}
  The same argument shows that
  \begin{equation}
    \label{serrevanishintrick}
    \HH^{(r^2-1)(g-1)-i}(\moduli_C(r,\mathcal{L}),\mathcal{W}_x\otimes \mathcal{W}_y^\vee\otimes \Theta^\vee)=0
   \end{equation}
  when~$(r^2-1)(g-1)-i \geq r^2+(r-1)$.

  Combining the above with \eqref{equation:sommese-vanishing}, we get that if~$(r^2-1)(g-1)\geq 2(r^2+r-1)$, then the cohomology~$\HH^i(\moduli_C(r,\mathcal{L}),\mathcal{V})$ vanishes for all~$i$. But this is satisfied for all~$g\geq 4$.
\end{proof}

\section{Concluding remarks}
\label{section:remarks}
\paragraph{The proofs of \cref{theorem:fully-faithful,theorem:sod}}
We can now explain how the vanishing theorems imply the fully faithfulness of the Fourier--Mukai functor, and how they give the semiorthogonal decomposition of \eqref{equation:sod}.

\begin{proof}[Proof of \cref{theorem:fully-faithful}]
  The Bondal--Orlov criterion from \cref{proposition:bondal-orlov} can be applied using \cref{theorem:fully-faithful-vanishing}.
\end{proof}

For the proof of \cref{theorem:sod} we need to use the following lemma.
\begin{lemma}
  \label{lemma:orthogonal-via-spanning-classes}
  Let~$X$ be a smooth projective variety. Let~$F\colon\mathcal{A}\hookrightarrow\derived^\bounded(X)$ and~$G\colon\mathcal{B}\hookrightarrow\derived^\bounded(X)$ be admissible embeddings. To check that~$\mathcal{A}$ is in the right orthogonal to~$\mathcal{B}$, it suffices to check this for spanning classes for~$\mathcal{A}$ and~$\mathcal{B}$, i.e.~whether
  \begin{equation}
    \Hom_{\derived^\bounded(X)}(G(T),F(S))=0
  \end{equation}
  for all objects~$S$ in a spanning class for~$\mathcal{A}$ and all objects~$T$ in a spanning class for~$\mathcal{B}$.
\end{lemma}
For sake of notational simplicity a spanning class will be closed under shifts.
\begin{proof}
  We need to check that
  \begin{equation}
    \Hom_{\derived^\bounded(X)}(G(B),F(A))=0
  \end{equation}
  for all~$A\in\mathcal{A}$ and~$B\in\mathcal{B}$. Applying the adjunction~$G\dashv G^{\mathrm{R}}$ this is equivalent to~$G^{\mathrm{R}}\circ F(A)$ being isomorphic to~0 for all~$A$. This in turn is implied by
  \begin{equation}
    \Hom_{\mathcal{B}}(T,G^{\mathrm{R}}\circ F(A))=0,
  \end{equation}
  where~$T$ runs over a spanning class for~$\mathcal{B}$. Now applying~$G\dashv G^{\mathrm{R}}$ and~$F^{\mathrm{L}}\dashv F$, this is equivalent to
  \begin{equation}
    \Hom_{\mathcal{A}}(F^{\mathrm{L}}\circ G(T),A)=0,
  \end{equation}
  which is equivalent to~$F^{\mathrm{L}}\circ G(T)$ being isomorphic to zero for all~$T$. This in turn is implied by
  \begin{equation}
    \Hom_{\mathcal{A}}(F^{\mathrm{L}}\circ G(T),S)=0,
  \end{equation}
  where~$S$ runs over a spanning class for~$\mathcal{A}$.
\end{proof}

\begin{proof}[Proof of \cref{theorem:sod}]
  The orthogonality criterion from \cref{lemma:orthogonal-via-spanning-classes} can be applied to the images of~$\derived^\bounded(C)$ under~$\Phi_{\mathcal{W}}$ and~$\Phi_{\mathcal{W}}\otimes\Theta^\vee$ by using the spanning class given by the skyscrapers, and using part~(3) of \cref{theorem:sod-vanishing}.

  The other orthogonality checks (there are~5~more) follow from Kodaira vanishing, and parts~(1) and~(2) of \cref{theorem:sod-vanishing}.
\end{proof}

\paragraph{On \cref{theorem:sod} for rank 2 and genus 2}
It is expected that in \cref{theorem:fully-faithful} the condition on the genus is not essential, and combining \cite[remark~5]{MR3713871} with \cite[theorem~1.1]{MR3764066} we know that the functor~$\Phi_{\mathcal{W}}$ is fully faithful for~$r=2$ and all~$g\geq 2$, so a posteriori we can conclude the vanishing results in \cref{theorem:fully-faithful-vanishing}.

But when~$r=2$ and~$g=2$ it is shown in \cite[theorem~2.9]{alg-geom/9506012} that \eqref{equation:start-of-sod} is the whole semiorthogonal decomposition, i.e.
\begin{equation}
  \derived^\bounded(\moduli_C(2,\mathcal{L}))=\left\langle
  \Theta^\vee,\mathcal{O}_{\moduli_C(2,\mathcal{L})},\Phi_{\mathcal{W}}(\derived^\bounded(C))
  \right\rangle.
\end{equation}
In particular, \cref{theorem:sod} cannot hold for~$r=2$ and~$g=2$, so we cannot have the vanishing result in part~3 of \cref{theorem:sod-vanishing}. We leave it to the interested reader to compute directly on the intersection~$Q_1\cap Q_2$ of~2~smooth~quadrics in~$\mathbb{P}^5$ that the sheaf cohomology of the tensor product of the restriction of (dual) spinor bundles twisted by~$\mathcal{O}_{Q_1\cap Q_2}(-1)$ is non-zero.

On the other hand, for all other combinations of rank and genus (and degree) it is expected that \cref{theorem:sod} holds.

\paragraph{Generalised Picard bundles}
The fully faithfulness result from \cref{theorem:fully-faithful} allows us to reprove known results on the inversion of generalised Picard bundles, and their deformation theory, originally proven in \cite{MR1309332,MR2193340}. This is remarked upon in \cite[remark~1]{MR3713871} in the case when~$r=2$. We will now give some details in the more general case here.

\begin{definition}
  \label{definition:generalised-picard}
  Let~$\mathcal{E}$ be a semistable vector bundle of rank~$n$ and degree~$e$, such that
  \begin{equation}
    re+n>rn(2g-2).
  \end{equation}
  Then the Fourier--Mukai transform~$\Phi_{\mathcal{W}}(\mathcal{E})$ is again a vector bundle, of rank~$re+n+rn(1-g)$, called a \emph{generalised Picard bundle}.
\end{definition}

The first result that follows from \cref{theorem:fully-faithful} is \cite[theorem~19]{MR2193340}. Recall that the kernel~$\mathcal{W}^{\mathrm{R}}$ for the right adjoint of~$\Phi_{\mathcal{W}}$ is given by~$\mathcal{W}^\vee\otimes p_1^*(\omega_C)[1]$. By fully faithfulness we have a natural equivalence~$\Phi_{\mathcal{W}^{\mathrm{R}}}\circ\Phi_{\mathcal{W}}\cong\identity_{\derived^\bounded(C))}$.
\begin{proposition}[Inversion formula]
  Let~$\mathcal{E}$ be a vector bundle as in \cref{definition:generalised-picard}. Then there exists an isomorphism
  \begin{equation}
    \mathcal{E}\cong\RR^1p_{1,*}\left( p_2^*(p_{2,*}(p_1^*(\mathcal{E})\otimes\mathcal{W})\otimes\mathcal{W}^\vee\otimes p_1^*(\omega_C) \right).
  \end{equation}
\end{proposition}

Similarly we can describe the deformation theory of generalised Picard bundles as in \cite[theorem~22]{MR2193340}. As the infinitesimal deformation theory of any sheaf~$\mathcal{E}$ on a smooth projective variety~$X$ is described by~$\HH^i(X,\sheafEnd(\mathcal{E}))$ for~$i=0,1,2$, the fully faithfulness of~$\Phi_{\mathcal{W}}$ gives the isomorphism in the next theorem.
\begin{proposition}
  Let~$\mathcal{E}$ be a vector bundle as in \cref{definition:generalised-picard}. Then there exists an isomorphism
  \begin{equation}
    \Ext_C^i(\mathcal{E},\mathcal{E})\cong\Ext_{\moduli_C(r,\mathcal{L})}^i(\Phi_{\mathcal{W}}(\mathcal{E}),\Phi_{\mathcal{W}}(\mathcal{E}))
  \end{equation}
\end{proposition}

\begin{proof}
  One uses that
  \begin{equation}
    \HH^i(C,\sheafEnd(\mathcal{E}))\cong\Ext_C^i(\mathcal{E},\mathcal{E})
  \end{equation}
  and
  \begin{equation}
    \HH^i(\moduli_C(r,\mathcal{L}),\sheafEnd(\Phi_{\mathcal{W}}(\mathcal{E})))\cong\Ext_{\moduli_C(r,\mathcal{L})}^i(\Phi_{\mathcal{W}}(\mathcal{E}),\Phi_{\mathcal{W}}(\mathcal{E})),
  \end{equation}
  and the left-hand sides are isomorphic by fully faithfulness.
\end{proof}

So the (infinitesimal) deformation theory of a generalised Picard bundle agrees with that of the original bundle. In particular, if~$\mathcal{E}$ is simple, then by Riemann--Roch the deformation space has dimension~$n^2(g-1)+1$.

\appendix

\section{An alternative proof of \texorpdfstring{\cref{proposition:determinant-of-cohomology}}{proposition \ref{proposition:determinant-of-cohomology}}}
\label{appendix:alternative-proof}
Because~$\moduli_C(r,\mathcal{L})$ is Fano, we have an isomorphism~$\Pic(C\times\moduli_C(r,\mathcal{L}))\cong\Pic C\oplus\mathbb{Z}$. Using this isomorphism we have
\begin{equation}
  \cc_1(\mathcal{W})=\cc_1(\mathcal{W}_x)+d.
\end{equation}

For a vector bundle~$\mathcal{E}$ we will use the following shorthand
\begin{equation}
  N_2(\mathcal{E})\coloneqq\cc_1(\mathcal{E})^2-2\cc_2(\mathcal{E}).
\end{equation}

We start with a preliminary lemma.
\begin{lemma}
  \label{lemma:equality}
  We have an equality
  \begin{equation}
    r\cdot p_{2,*}(N_2(\mathcal{W}))=-2[\Theta]+2d\cc_1(\mathcal{W}_x).
  \end{equation}
\end{lemma}

\begin{proof}
  By \cite[theorem~1]{MR0325615} we have that
  \begin{equation}
    \begin{aligned}
      -2[\Theta]
      &=[\omega_{\moduli_C(r,\mathcal{L})}] \\
      %&=\cc_1\left( (\det\RR^1p_{2,*}\adjoint\mathcal{W})^\vee \right) \\
      %&=\cc_1\left( \det\RR^0p_{2,*}(\adjoint\mathcal{W})\otimes(\det\RR^1p_{2,*}\adjoint\mathcal{W})^\vee \right) \\
      &=\cc_1\left( \det\RRR p_{2,*}\adjoint\mathcal{W} \right).
    \end{aligned}
  \end{equation}
  Using Grothendieck--Riemann--Roch we can further rewrite this to
  \begin{equation}
    \begin{aligned}
      \cc_1\left( \det\RRR p_{2,*}\adjoint\mathcal{W} \right)
      &=\chern(p_{2,!}\adjoint\mathcal{W})_{\deg=1} \\
      &=p_{2,*}\left( (\chern(\adjoint\mathcal{W})\otimes p_1^*\todd_C)_{\deg=2} \right) \\
      &=p_{2,*}\left( (\chern(\sheafEnd(\mathcal{W}))\otimes p_1^*\todd_C)_{\deg=2} \right).
    \end{aligned}
  \end{equation}
  Restricting ourselves to the terms that contribute to the part in degree~2 we obtain
  \begin{equation}
    \begin{aligned}
      &p_{2,*}\left( (\chern(\sheafEnd(\mathcal{W}))\otimes p_1^*\todd_C)_{\deg=2} \right) \\
      &\quad=p_{2,*}\left( \left( \left( r^2+\frac{1}{2}N_2(\mathcal{W}\otimes\mathcal{W}^\vee) \right)\left( 1+\frac{1}{2}p_1^*\cc_1(\tangent_X) \right) \right)_{\deg=2} \right) \\
      &\quad=p_{2,*}\left( \frac{1}{2}N_2(\mathcal{W}\otimes\mathcal{W}^\vee) \right) \\
      &\quad=p_{2,*}\left( \frac{r}{2}\left( N_2(\mathcal{W})+N_2(\mathcal{W}^\vee) \right) + \cc_1(\mathcal{W})\cc_1(\mathcal{W}^\vee) \right) \\
      &\quad=p_{2,*}\left( rN_2(\mathcal{W})-\cc_1(\mathcal{W})^2 \right).
    \end{aligned}
  \end{equation}
\end{proof}

We will now apply the normalisation for the Poincar\'e bundle~$\mathcal{W}$.

\begin{proof}[Proof of \cref{proposition:determinant-of-cohomology}]
  By Grothendieck--Riemann--Roch we get
  \begin{equation}
    \begin{aligned}
      \cc_1(\det\RRR p_{2,*}\mathcal{W}^\vee)
      &=\chern(p_{2,!}\mathcal{W}^\vee)_{\deg=1} \\
      &=p_{2,*}\left( \chern(\mathcal{W}^\vee)p_1^*\todd_C \right) \\
      &=p_{2,*}\left( \left( r-\cc_1(\mathcal{W})+\frac{1}{2}N_2(\mathcal{W}) \right)\left( 1+\frac{1}{2}p_1^*\cc_1(\tangent_C) \right)_{\deg=2} \right) \\
      &=p_{2,*}\left( \frac{1}{2}\cc_1(\mathcal{W})p_1^*\cc_1(\mathrm{K}_C) \right) + \frac{1}{2}p_{2,*}N_2(\mathcal{W}) \\
      &=(g-1)\cc_1(\mathcal{W}_x)+\frac{1}{r}\left( -[\Theta]+d\cc_1(\mathcal{W}_x) \right)
    \end{aligned}
  \end{equation}
  where in the last step we used \cref{lemma:equality}. Now \cref{proposition:determinant-of-cohomology} follows from \eqref{equation:normalisation}.
\end{proof}

\printbibliography

\end{document}